\documentclass{article}

\usepackage{graphicx}
\usepackage{amsthm,xcolor}
\usepackage{amsfonts}
\usepackage{amsmath}
\usepackage{amscd}
\usepackage{amssymb}
\usepackage{alltt}
\usepackage{url}
\usepackage{ellipsis}

\newtheorem{theorem}{Theorem}
\newtheorem*{lemma*}{Lemma}
\newtheorem{lemma}{Lemma}

\newtheorem*{definition*}{Definition}
\newtheorem*{prop*}{Proposition}

\theoremstyle{remark}
\newtheorem{remark}[equation]{Remark}

\newcommand{\uu}{\rule[-0.1ex]{0.4em}{0.15ex}\rule[0ex]{0.08em}{0ex}}
\renewcommand{\_}{\uu}
\newcommand{\bquote}{^\backprime}
\newcommand{\contribchapter}[2]{}
\newcommand{\cmq}{\op{cm}_\phi\bquote}
\newcommand{\op}[1]{\hbox{#1}}
\newcommand{\Hp}{\hbox{Hp}}
\newcommand{\notempty}{\varnothing\!\!\ne}

\title{Simple Matroids and Alfred North Whitehead's theory of dimension (1906)}
\author{Thomas C. Hales}
\date{}

\begin{document}
\maketitle


\def\cl{\op{cl}} 
\def\dimq{\dim_\phi\bquote}

\contribchapter{Simple Matroids and\\Alfred North Whitehead's%
\\Theory of Dimension (1906)}{Thomas C. Hales}
\bigskip

\begin{abstract} We give a correspondence between simple matroids and a
reconstruction of Whitehead's theory of dimension, as found
in \emph{On Mathematical Concepts of the Material World} (MW), published
in 1906~\cite{Whitehead1906}.  In brief, if $(E,\phi)$ is a
geometrical system in the generalized sense of Whitehead, where $E$ is
finite and $\phi$-maximal, then $E$ is a simple matroid.
(\emph{Generalized} here means that Whitehead's three-dimensional
axiom has been replaced with finite-dimensionality.) Within the context
of the generalized geometrical axioms, \emph{$\phi$-maximal} subsets
correspond with simple submatroids.  Conversely, every simple matroid
is a $\phi$-maximal geometrical system in the generalized sense of
Whitehead.
\end{abstract}

\bigskip

Matroids were introduced in 1935 by Hassler Whitney~\cite{Whitney1935}
and independently by Takeo Nakazawa the same
year~\cite{NishimuraKuroda_Nakazawa}.  As originally conceived by
Whitney, matroid theory was a common generalization of the notion of
linear independence in vector spaces and circuits in graphs.  Nakazawa
gave axioms for linear dependence: properties that abstractly describe
sequences of vectors $a_1,\ldots,a_s$ whose product $a_1a_2\cdots
a_s=0$ is zero in the exterior algebra of a vector space.  In 1936,
Nakazawa connected his theory explicitly with Whitney's.  The history
of matroid theory is documented in~\cite{Kung1986SourceBook}, but the
book does not mention MW.

Whitehead's theory is a precursor to their work, although there is no
evidence of any direct influence of Whitehead on Whitney or Nakazawa.
Some social connections between Whitehead and Whitney's circle of
associates at Harvard are mentioned in a remark at the end of this
article, but these might amount to no more than a historical curiosity.

According to Whittaker, MW is ``extremely difficult reading, the
propositions being established by aid of the Peano symbolism: and it
fell completely flat so far as enlisting the attention of other
writers was concerned''~\cite{Whittaker1947}.  Indeed, much of MW is
written in the unwelcoming symbolic notation of Whitehead and
Russell's \emph{Principia Mathematica}.  Yet Whitehead held MW in
surprisingly high esteem. ``Looking backward thirty years later, when
he was a famous philosopher at Harvard, he [Whitehead] told me [Lowe]
that insofar as he could lay claim to any originality, he
thought \emph{`On Mathematical Concepts of the Material World'} was
the most original thing he had done''~\cite[v1~p296]{Lowe1}.
Whitehead also stated that his theory of dimension was one of the main
objects of the memoir.

\bigskip

Familiarity with MW and Oxley's book on matroids is
assumed~\cite{Oxley2011}.  Whitehead's notation will be used for his
theory, and Oxley's notation will be used for matroids.  In what
follows, the notation O.X refers to numbered results in Oxley, and
notation W.X refers to numbered results in Whitehead's MW.

In this article, all matroids are assumed to have a finite ground set.
The definition of a matroid can be presented
in several ways -- in terms of flats, a closure operator, a rank
function, or independent sets.  All of the definitions turn out to be
equivalent, and the equivalence of these definitions (among others
such as circuits) is the starting point of matroid theory.

Whitehead's MW has an analogue of each of these fundamental concepts
of matroid theory.  His name for flats and closure was the
$\phi$-common region; his name for rank was the $\phi$-dimension
number; and his name for independent set was $\phi$-axial.  Both
$\phi$-maximality and the simple matroid property are hereditary under
passage to subsets.  Within the context of his geometrical axioms,
$\phi$-maximal subsets correspond with simple submatroids.  A
significant part of MW is devoted to proving propositions about these
concepts and their relationships.


The analogy between simple matroids and Whitehead's theory is shown in
the following table. (The backquote $\bquote$ denotes function
application and is a peculiarity of Whitehead's notation that can be
disregarded here.)

\begin{center}
\begin{tabular}{ll}
{\bfseries Simple Matroid} & {\bfseries Whitehead} \\
closure operator & $\phi$-common region of $u$, or $\cmq{u}$\\
having the same closure & $\phi$-equivalent\\
rank & $\phi$-dimension number of $u$, or $\dimq{u}$\\
independent set & $\phi$-axial (or $\phi$-prime in a $\phi$-maximal subset)\\
simple matroid & $\phi$-maximal (in the presence of other axioms)
\end{tabular}
\end{center}

\bigskip

Whitehead's theory includes examples that lie outside the scope of
simple matroid theory.  It is only by imposing extra conditions on
Whitehead's theory -- that $E$ is finite and $\phi$-maximal -- that we
obtain a correspondence with simple matroid theory.  Yet even without
these extra conditions, numerous propositions in MW have close
analogues in simple matroid theory.  Theorems~\ref{thm:E}
and \ref{thm:2} give the relationship between simple matroids and a
reconstruction of Whitehead's theory.  We modify Whitehead's original
axioms by replacing his three-dimensional axiom $(\nu)$ with an axiom
of finite-dimensionality $(\nu')$, as follows.

\begin{definition*}[Whitehead's generalized geometrical axioms] Let $E$ be a set, and
let some of the subsets of $E$ be designated as $\phi$-classes.  We
say that $(E,\phi)$ is \emph{geometrical (or a geometrical system) in
the generalized sense of Whitehead} if the following five axioms hold.
\begin{itemize}
\item[$(\lambda)$] $E$ is a $\phi$-class.
\item[$(\mu)$] For all $x\in E$, the singleton $\{x\}$ is a $\phi$-class.
\item [$(\nu')$] The $\phi$-dimension number of $E$ is finite, and 
$\cmq\emptyset=\emptyset$.
\item[$(\pi)$] For all $u\subseteq E$ and for all $\phi$-axial subsets $v$ of 
its $\phi$-common region $\cmq u$,
there exists a subset $w\subseteq E$ such that $v\cup w$ is 
$\phi$-axial and $\phi$-equivalent
to $u$.
\item [$(\rho)$] For all $\phi$-axial subsets $u,v\subseteq E$, if $|u\cap v|\ge2$,
then $u\cup v$ is $\phi$-maximal.
\end{itemize}
\end{definition*}

\begin{remark} 
Whitehead defines $\phi$ to be \emph{geometrical} if it satisfies five
axioms $(\lambda,\mu,\nu,\pi,\rho)\_\Hp\_\phi$ (see \S{W}.10).  The
axioms $\lambda$, $\mu$, $\pi$ and $\rho$ are Whitehead's axioms
$\lambda\_\Hp\_\phi$, $\mu\_\Hp\_\phi$, $\pi\_\Hp\_\phi$, and
$\rho\_\Hp\_\phi$.  ($\Hp\_\phi$ is Whitehead's abbreviation for
\emph{hypothesis on $\phi$-classes}.)

Whitehead's axiom $\nu\_\Hp\_\phi$ is that $E$ is three-dimensional.
He makes this assumption because his interest is three-dimensional
geometry, while observing that ``the reasoning can be applied to
higher dimensions, only more elaborate inductions and an extra axiom
are required.'' Whitehead does not spell out this extra axiom.  We
have replaced his axiom with finite-dimensionality $(\nu')$.
Whitehead proves that if $|E|\ne1$, then $\cmq\emptyset=\emptyset$
(see Prop. W.11.12, whose proof assumes $|E|\ne1$). We have included
this identity as part of axiom $(\nu')$ to avoid exceptions when $E$
is a singleton. The identity is used in the proof of no loops
in the associated matroid.  When finite-dimensionality $(\nu')$
replaces three-dimensionality $(\nu\_\Hp\_\phi)$, we refer to the
geometrical axioms in the \emph{generalized} sense of Whitehead.
\end{remark}


The first theorem states that every finite $\phi$-maximal geometrical
system in the generalized sense of Whitehead is a simple matroid.

\begin{theorem}\label{thm:E}
Let $(E,\phi)$ be geometrical in the generalized sense of Whitehead.  If $E$ is
finite and $\phi$-maximal, then $E$ is a simple matroid whose closure
operator is $u\mapsto\cmq{u}$.  Moreover, a subset of $E$ is a flat of
the matroid iff it is an intersection of $\phi$-classes.  The
independent sets of the matroid are the $\phi$-axial sets (together
with the empty set).  Finally, the rank function on the matroid is
equal to the $\phi$-dimension number on every nonempty set.
\end{theorem}

The proof of the theorem relies on new inductive proofs of some
propositions that were limited in Whitehead to three dimensions.  The
proof of the theorem is given below.  First, we make a series of remarks.

\begin{remark}
Two different families of $\phi$-classes on $E$ determine the same
matroid iff they determine the same collection of flats; that is, they
both give the same collection of intersections.
\end{remark}

\begin{remark}
In an elucidatory note, Whitehead indicates that $\phi$-classes are
intended to generalize the geometrical idea of flatness in Euclidean
space (see W.3.12).  The use of the term \emph{flat} in both Whitehead
and matroid theory points to a shared intuitive understanding of what
the theory is about.
\end{remark}

\begin{remark}
In Whitehead's terminology, a nonempty set is $\phi$-prime if it does
not have a proper subset with the same $\phi$-common region.  A
$\phi$-prime set is $\phi$-axial if it has the largest cardinality
among $\phi$-prime sets with the same $\phi$-common region.  The
$\phi$-dimension number of a set is defined to be the cardinality of a
$\phi$-axial set with the same $\phi$-common region.  Whitehead's
theory includes examples such as a pair of skew lines: a set with two
elements whose affine closure is three-dimensional; a pair of skew
lines is $\phi$-prime but not $\phi$-axial.  In a matroid, there is no
distinction between (the analogues of) $\phi$-prime and $\phi$-axial;
that is, all bases of the same flat have the same cardinality.
Whitehead's $\phi$-maximal condition rules out examples such as skew
lines to ensure a match between the $\phi$-dimension number and the
cardinality of $\phi$-prime sets.
\end{remark}

\begin{remark} 
Whitehead's definitions can also be applied to infinite sets or even
proper classes.  If $E$ is finite or infinite, and if every subset of
$E$ is declared to be a $\phi$-class, then the $\phi$-dimension number
of every nonempty subset $u\subseteq E$ is equal to its cardinality.
\end{remark}

\begin{remark}
Since Whitehead does not assume that $E$ is finite, simple matroids of
finite rank might be a better fit to Whitehead's geometrical axioms
than simple matroids on a finite ground set.  This is a matter for
future investigation.  However, we confine our attention to finite
ground sets.
\end{remark}

\begin{remark}
To give more historical context, in his earlier \emph{A Treatise on
Universal Algebra} (1898), Whitehead gave an exposition of properties
of linear dependence and independence, including dependence and
independence in the exterior
algebra~\cite[\S64,\S91--\S96]{Whitehead1898}. The propositions
include the statement that the maximum number $n$ of independent
elements is one more than the dimension of the corresponding
projective space (called the ``positional manifold'') and that any set
of $n$ (but no fewer) independent elements span.
\end{remark}

\begin{remark} 
Assuming the geometrical axioms, by Prop. W. 12.12, every nonempty
subset of $E$ has a set of $\phi$-axes.  This implies that the
$\phi$-dimension number is defined (and is a cardinal number) for
every nonempty subset of $E$. (This remark is needed, because
Whitehead's general specification of $\phi$-dimension skirts the issue
of its definedness.)
\end{remark}

\begin{remark}
Our proofs will use Whitehead's propositions, which are consequences
of his five axioms of geometry, but it can be checked that whenever we
invoke Whitehead's propositions, the proofs go through using only
$\lambda,\mu,\nu',\pi,\rho$.  In fact, we will never need to use the
axiom $\rho$.  (However, Theorem~\ref{thm:E} assumes that $E$ is
$\phi$-maximal, which is a strong form of axiom $\rho$.)  If the
proposition number is W.X with $X<10$, then the proposition occurs
before the statement of the axioms in section 10 of MW, and there is
nothing to check.  The propositions in the early sections of Whitehead
include results that correspond in matroid theory with the basic
properties of the closure operator of matroids, the hereditary
property of independent sets, the relation between the independent
sets and closure, and the simplicity property.
\end{remark}

We will repeatedly use the following two propositions from Whitehead,
which give relations between the $\phi$-dimension number and
$\phi$-common region.  

\begin{prop*} {\bfseries W.12.21.}\quad 
Let $(E,\phi)$ be a geometrical system in the generalized sense of Whitehead.
Assume that $u\subseteq E$, $\notempty{v}\subseteq E$, and
$\cmq{v}\subseteq \cmq{u}$.  Then $\dimq{v}\le \dimq{u}$.
\end{prop*}

\begin{proof} See MW.
\end{proof}

In particular, $\dimq{E}<\infty$ is an upper bound on the
$\phi$-dimension number of nonempty subsets of $E$.

\begin{prop*}{\bfseries W.12.23.}\quad 
Let $(E,\phi)$ be a geometrical system in the generalized sense of Whitehead.
Assume that $u\subseteq E$, $\notempty{v}\subseteq E$, and
$\cmq{v}\subseteq \cmq{u}$.  Then 
\[
\cmq{v}= \cmq{u}\quad \hbox{iff}\quad\dimq{v}= \dimq{u}.
\]
\end{prop*}

\begin{proof} See MW.
\end{proof}

Before proving the theorem, we give two key propositions (Props. W.12.37 and
W.13.11) that appear in MW whose proofs need to be updated to our
assumption context.  As mentioned above, Whitehead observed ``the
reasoning can be applied to higher dimensions, only more elaborate
inductions~\ldots\ are required.'' Here we supply the more elaborate
inductions.  We add Lemma~\ref{lemma}, which is not found in
Whitehead, that will be used three times to update Whitehead's
proofs. The lemma serves as an induction step in other proofs. Prop. W.12.37
gives the existence of a $\phi$-prime subset of a given nonempty
set. Prop. W.13.11 is the hereditary property of $\phi$-maximal subsets.
It will be used to show in the proof of
Theorem~\ref{thm:E} that the $\phi$-prime sets are in fact
$\phi$-axial.

\begin{prop*} {\bfseries W.12.37.}\quad
Let $(E,\phi)$ be a geometrical system in the generalized sense of Whitehead.  Let
$\emptyset\ne u\subseteq E$.  Then there exists a (finite)
$\phi$-prime set $u'\subseteq u$ that is $\phi$-equivalent to $u$.
\end{prop*}

\begin{proof} Assume $\emptyset\ne u\subseteq E$.  We claim that if
$v\subseteq u$ is any finite subset of $u$ of maximal dimension among
finite subsets of $u$, then $v$ is $\phi$-equivalent
to $u$.  We prove the contrapositive of the claim.  Assume that $v$ is
a finite subset of $u$ that is not $\phi$-equivalent to $u$.  Pick
$x\in{u}\setminus\cmq{v}$ (using Props. W.4.21, W.4.25).  Then
$\cmq{v}\subsetneq\cmq(v\cup\{x\})$ and
$\dimq{v}<\dimq(v\cup\{x\}) \le\dimq u$ (by Props. W.12.21, W.12.23).
This shows that $v$ does not have maximal dimension among finite
subsets of $u$.  This proves the
claim.

Let $v$ be a finite subset of $u$ that is $\phi$-equivalent to $u$.
Let $u'\subseteq v$ be a subset of $v$ of minimal cardinality among
those subsets of $v$ that are $\phi$-equivalent to $u$.  Then $u'$ is
clearly $\phi$-prime and $\phi$-equivalent to $u$.
\end{proof}

\begin{lemma}\label{lemma}
Let $(E,\phi)$ be a geometrical system in the generalized sense of Whitehead.  Let
$v\subseteq E$ be $\phi$-prime, and let $x\in E$ be such that
$v\cup\{x\}$ is $\phi$-maximal and such that $x\not\in\cmq{v}$.  Then
$v$ and $v\cup\{x\}$ are $\phi$-axial.  Also
$\dimq(v\cup\{x\})=\dimq{v}+1$.
\end{lemma}

\begin{proof}
Choose $v'\subseteq{v}\cup\{x\}$ that is $\phi$-equivalent to
$v\cup\{x\}$ and $\phi$-prime.  By the definition of $\phi$-maximal,
we have that $v'$ is $\phi$-axial.  Also, $v'$ is finite.  We have
\[
|v|\le\dimq{v}<\dimq(v\cup\{x\})=\dimq{v'}=|v'|\le|v|+1.
\]
All weak inequalities must be equalities. This implies that
$|v|=\dimq{v}$, that $v'=v\cup\{x\}$, and that
$|v|+1=|v\cup\{x\}|=\dimq(v\cup\{x\})$.  
The conclusion follows.
\end{proof}

\begin{prop*} {\bfseries W.13.11.}\quad
Let $(E,\phi)$ be a geometrical system in the generalized sense of Whitehead.
Every nonempty subset of a $\phi$-maximal set is also $\phi$-maximal.
\end{prop*}

\begin{proof}
Assume $u$ is $\phi$-maximal.  We argue by contradiction, and assume
that $u$ has a nonempty subset that is not $\phi$-maximal.  Among all
such nonempty subsets of $u$ we choose a subset $u'$ whose
$\phi$-dimension number is as large as possible.  We have
$\dimq{u'}<\dimq{u}$ (by Props. W.12.21, W.12.23, W.4.21).  Let $v$ be
any subset of $u'$ that is $\phi$-prime and $\phi$-equivalent to $u'$
(which exists by W.12.37). Let $x\in{u}\setminus\cmq{v}$ (this
exists). Then $\dimq{v}<\dimq(v\cup\{x\})$, and by the dimension
maximizing choice of $u'$ (and $v$), we have that $v\cup\{x\}$ is
$\phi$-maximal.  By Lemma~\ref{lemma}, we have that $v$ is
$\phi$-axial.  This shows that $u'$ is $\phi$-maximal, which is
contrary to the choice of $u'$.
\end{proof}

\begin{proof} \emph{(Theorem. \ref{thm:E}).}\quad
Let $(E,\phi)$ be geometrical in the generalized sense of Whitehead.  Assume that
$E$ is finite and $\phi$-maximal.  It follows from Prop. W.13.11 that
every nonempty subset of $E$ is $\phi$-maximal.  In particular, a
subset of $E$ is $\phi$-prime iff it is $\phi$-axial.  We refer to
this particular fact as property $(\tau)$.

Let ${\mathcal I}$ be the collection of $\phi$-prime (or $\phi$-axial)
subsets of $E$ together
with the empty set.  To show that $(E,{\mathcal I})$ is a matroid, we verify
the three axioms \S{O}.1.1 of independent sets.

\begin{itemize} 
\item[{\bfseries I1}] $\emptyset\in{\mathcal I}$: 
This is true by the choice of ${\mathcal I}$.
\item[{\bfseries I2}] We claim that if $I\in{\mathcal I}$ and $I'\subseteq I$, 
then $I'\in{\mathcal I}$.
This is Whitehead's Proposition W.5.23. 
(Whitehead calls Prop. W.5.23 and two closure laws ``the foundation of the whole theory.  
It is remarkable that it requires no axiom concerning $\phi$.'')
\item [{\bfseries I3}] Assume that $I_1$ and $I_2$ are in ${\mathcal I}$ and 
$|I_1|<|I_2|$. We claim that there exists $x\in I_2\setminus{I_1}$
such that $I_1\cup\{x\}\in {\mathcal I}$. The case $I_1=\emptyset$
being trivial (by Prop. W.5.233), we assume that $I_1\ne\emptyset$.
\emph{We claim that $I_2$ is not a subset of $\cmq{I_1}$.}  Otherwise, for a 
contradiction, we have $\cmq{I_2}\subseteq \cmq{I_1}$ (by Prop. W.4.21, W.4.31),
and
\[
|I_2|=\dimq{I_2}\le\dimq{I_1}= |I_1|
\] 
(by property $\tau$ and W.12.21), which is contrary to assumption.
Thus, by the claim, there exists $x\in I_2\setminus \cmq{I_1}$.  By
W.4.25, we have $x\in{I_2\setminus{I_1}}$.  Recall from above that
$I_1\cup\{x\}$ is $\phi$-maximal.  Apply Lemma~\ref{lemma} with
$v=I_1$.  Hence $I_1\cup\{x\}$ is $\phi$-prime.  By definition
$I_1\cup\{x\}\in{\mathcal I}$.  This completes the proof of the third
axiom of matroids.
\end{itemize}
This shows that $(E,{\mathcal I})$ is a matroid.

Furthermore, the matroid $(E,{\mathcal{I}})$ has no loops by
Prop. W.5.233.  Furthermore, for all $x,y\in{E}$, we have
$\{x,y\}\in\mathcal{I}$ by Prop.~W.5.235 and property $(\tau)$.
Hence, by the definition of simplicity (Oxley, p12), the matroid is
simple.

Next we show that the function $u\mapsto \cmq{u}$ is the closure
operator associated with the simple matroid $(E,{\mathcal I})$.  To
accomplish this, we show that $\op{cm}_\phi$ satisfies the
characteristic properties of a closure operator in Lemma~O.1.4.3 and
Theorem~O.1.4.5.

\begin{itemize}
\item[{\bfseries CL0}] $u\in{\mathcal I}$ iff 
(for all $x\in u$, $x\notin \cmq{(u\setminus\{x\})}$); by Prop. W.5.231.
\item[{\bfseries CL1}] If $u\subseteq E$, then $u\subseteq \cmq{u}$ by Prop. W.4.25.
\item[{\bfseries CL2}] If $u\subseteq v\subseteq E$, 
then $\cmq{u}\subseteq\cmq{v}$ by Prop. W.4.21.
\item[{\bfseries CL3}] If $u\subseteq E$, then $\cmq\cmq{u}=\cmq{u}$ by Prop. W.4.31.
\item[{\bfseries CL4}] 
Assume $u\subseteq E$ and $x\in E$ and
$y\in \cmq{(u\cup\{x\})}\setminus\cmq{u}$. We claim
$x\in\cmq{(u\cup\{y\})}$.  Dismissing a trivial case, we may assume that
$u$ is nonempty.  By Prop. W.12.37, we may choose a $\phi$-prime set
$I_1\subseteq u$ that is $\phi$-equivalent to $u$. Then by
Props. W.4.43, W.4.21, we have
$y\in \cmq{(I_1\cup\{x\})}\setminus\cmq{I_1}$.  We have
$x\not\in\cmq{I_1}$. By Lemma~\ref{lemma}, $I_1\cup\{x\}$ is
$\phi$-axial, and $\dimq(I_1\cup\{x\}) = \dimq{I_1}+1$.  We have
\[
\cmq{I_1}\subsetneq\cmq{(I_1\cup\{y\})}\subseteq\cmq{(I_1\cup\{x\})}.
\]
Comparing dimensions (using Prop. W.12.21 and W.12.23), we find that
$I_1\cup\{y\}$ and $I_1\cup\{x\}$ have the same $\phi$-dimension
number and are therefore $\phi$-equivalent. This gives
\[
x \in \cmq{(I_1\cup\{x\})}=\cmq{(I_1\cup\{y\})}\subseteq\cmq{(u\cup\{y\})}.
\]
This completes the proof of CL4.  
\end{itemize}
Hence $\op{cm}_\phi$ is the closure operator $\cl$ of the matroid.

The flats of the matroid are those of the form $\cmq{u}$.  By
Whitehead's definition of the $\phi$-common region, a set is a
$\phi$-common region iff it is an intersection of $\phi$-classes.

Finally, we show that the $\phi$-dimension number is the rank of the
matroid on nonempty sets.  (In Whitehead, the $\phi$-dimension number
of the empty set is undefined.)  Let $u$ be nonempty.  Both the rank
$r$ and the $\phi$-dimension number satisfy the closure property:
\[
r(u) = r(\cl(u)),\quad \dimq(u) = \dimq(\cl(u)),
\]
by Lemma~O.1.4.2 and by the definitions of $\phi$-dimension number and
$\phi$-equivalence.  Thus, it is enough to check that the rank is
equal to the $\phi$-dimension number when $u=I\in{\mathcal I}$ is
$\phi$-prime and hence also $\phi$-axial. In this case, $I$ is an
independent set, and
\[
r(I) = |I| = \dimq{I},
\]
by Prop. O.1.3.5 and by the definition of $\phi$-axial.  This proves
that the rank is equal to the $\phi$-dimension number.
\end{proof}

The converse theorem states that every simple matroid is a finite
$\phi$-maximal geometrical system in the generalized sense of
Whitehead.  To be explicit, the converse theorem chooses the
$\phi$-classes to be flats, restricts to finite simple matroids,
replaces Whitehead's dimension axiom with finite-dimensionality
$(\nu')$, and avoids some comparisons of the theories on the
empty set.

\begin{theorem}\label{thm:2} Let $(E,{\mathcal I})$ be a simple matroid,
with finite ground set $E$ and independent sets ${\mathcal I}$.
Define the $\phi$-classes to be the flats (i.e. closed sets) of the
matroid.  Then
\begin{enumerate}
\item For any subset $u$ of the ground set $E$, the $\phi$-common region of
  $u$ is the matroid closure of $u$.
\item Two subsets $u,v$ of the ground set
$E$ are $\phi$-equivalent iff they have the same matroid closure.  
\item A nonempty subset of the ground set $E$ is $\phi$-prime iff
  it is an independent set of the matroid.
\item The $\phi$-dimension number of a nonempty subset of the ground set
  $E$ is its matroid rank.
\item If $E$ is nonempty, then $E$ is $\phi$-maximal. 
Moreover, every $\phi$-prime set is $\phi$-axial.
\item $(E,\phi)$ is geometrical in the generalized sense of Whitehead.
\end{enumerate}
\end{theorem}

\begin{remark} Statement (5) shows why it is necessary
to introduce the assumption that
$E$ is $\phi$-maximal in Theorem~\ref{thm:E}.
Statement (6) is the main conclusion of the converse theorem,
showing that every simple matroid satisfies the generalized axioms of Whitehead.
\end{remark}

\begin{proof} 1. Let $u$ be a subset of the ground set $E$.
By (Oxley, p31, Ex.1c), $\cl(u)$ is the intersection of flats
containing $u$; and hence by definition, $\cl(u)=\cmq{u}$.

2. This statement follows directly from the definition of
$\phi$-equivalence and statement (1).

3. Compare Theorem~O.1.4.5 with Proposition W.5.231, which give the
same criterion for an independent set (resp. nonempty $\phi$-prime) in
terms of matroid closure (resp. $\phi$-common region). Compare CL0
above.

4. Let $u$ be a nonempty subset of the ground set.  Since the rank and
the $\phi$-dimension number both satisfy (O.1.4.2)
\[
r(u) = r(\cl(u)),\qquad \dimq u = \dimq(\cl(u)),
\]
we may assume without loss of generality that $u$ is a flat.  Let
${\mathcal P}$ be the set of subsets $I$ of $u$ that are $\phi$-prime
and $\phi$-equivalent to $u$.  We have $I\in{\mathcal P}$ iff $I$ is
an independent set whose closure is $u$.  The set ${\mathcal P}$ is
nonempty (i.e. a basis of $u$ exists by O.1.4.10), and all members of
${\mathcal P}$ have the same cardinality.  If $I\in{\mathcal P}$,
\[
r(u) = r(\cl(I)) = r(I) = |I| = \max_{I\in{\mathcal P}} |I| = \dimq u.
\]

5. We show that every $\phi$-prime set is $\phi$-axial.
Let $I$ be $\phi$-prime. Then
\[
|I| = r(I) = \dimq(I).
\]
Hence $I$ is $\phi$-axial. This implies that every nonempty subset of
$E$ is $\phi$-maximal.

6.$\lambda$. $E$ is closed by CL1 of O.1.4.3.

6.$\mu$ and 6.$\nu'$.  Every finite matroid clearly has finite
dimension.  In a simple matroid, singletons and the empty set are
closed, by the circuit test of Prop. O.1.4.11.ii.

6.$\pi$.  Let $u\subseteq E$ be closed (without loss of generality),
and let $I$ be a nonempty independent set contained in $u$.  Let $J$
be a basis of $u$ containing $I$.  The closure of $J$ is $u$, and $J$
is $\phi$-equivalent to $u$.

6.$\rho$. From (5), every nonempty subset of $E$ is $\phi$-maximal.
\end{proof}

\begin{remark} As a historical aside, we mention some significant
social connections between Whitehead and Whitney's circle of
associates, as mathematicians in the same intellectual milieu.  Of
course, these social ties have no bearing on the validity of the
theorems proved above.  Both of them were at Harvard: Whitney as a
young PhD (received in 1932), and Whitehead from 1924 until his
retirement in 1937.  Whitney's PhD advisor was George Birkhoff, who
had numerous interactions with Whitehead at Harvard.  In fact, the
Whiteheads lived two floors above the Birkhoffs on Memorial Drive, and
they ``were very congenial with
them''~\cite{birkhoff1989mathematics}. Harvard's Society of Fellows
grew out of a conversation between Whitehead and the biochemist
L. J. Henderson (who, years earlier, had helped to recruit Whitehead
to Harvard)~\cite[v2~p254]{Lowe1}.  When the Society was organized in
1933, George Birkhoff's son Garrett became one of the inaugural Junior
Fellows, and Whitehead became an inaugural Senior Fellow. Garrett
Birkhoff was an early contributor to matroid theory, publishing the
same year as Whitney in 1935~\cite{Birkhoff1935}.  Incidentally,
another of the six inaugural Junior Fellows was Quine, who received
his PhD under Whitehead at Harvard in 1932.  The Society of Fellows
met for dinner Monday nights. It is noteworthy but possibly merely
coincidental that Whitehead and Garrett Birkhoff were in regular
contact at dinners during the period when Birkhoff was writing on
matroids.
\end{remark}




\bigskip
In summary and conclusion, Alfred North Whitehead introduced a theory
of dimension in 1906 that is a precursor to the theory of simple matroids.
By generalizing his theory to finite dimensions and restricting his
theory to finite $\phi$-maximal sets, a multi-faceted correspondence
is obtained with simple matroid theory.  That correspondence
connects notions in Whitehead's theory with notions in matroid
theory such as rank, independent sets, and closure.

\newpage
\bibliographystyle{plain}
 \bibliography{ref} 
\end{document}